\documentclass[reqno]{amsart}

\usepackage[utf8]{inputenc}
\usepackage[T1]{fontenc}
\usepackage{lmodern}
\usepackage[top=3cm,bottom=3cm,left=3cm,right=3cm]{geometry}
\usepackage{amsmath,amssymb,amsthm,mathtools}
\usepackage{enumitem}
\usepackage{xcolor}
\usepackage{hyperref}
\usepackage{microtype}

\hypersetup{
  colorlinks=true,
  linkcolor=blue!60!black,
  citecolor=blue!60!black,
  urlcolor=blue!60!black
}

\numberwithin{equation}{section}

\newtheorem{theorem}{Theorem}[section]
\newtheorem{proposition}[theorem]{Proposition}
\newtheorem{lemma}[theorem]{Lemma}
\newtheorem{corollary}[theorem]{Corollary}

\newtheorem*{maintheorem}{Main Theorem}
\newtheorem*{compactrigidity}{Compact Rigidity}

\theoremstyle{definition}
\newtheorem{definition}[theorem]{Definition}

\theoremstyle{remark}
\newtheorem{remark}[theorem]{Remark}

\newcommand{\R}{\mathbb R}

\newcommand{\Z}{\mathbb Z}
\newcommand{\T}{\mathbb T}

\DeclareMathOperator{\Aut}{Aut}
\DeclareMathOperator{\Aff}{Aff}
\DeclareMathOperator{\GL}{GL}
\DeclareMathOperator{\Lie}{Lie}
\DeclareMathOperator{\MaxK}{MaxK}

\newcommand{\id}{\mathrm{id}}

\begin{document}

\title [Solvability and Rigidity for Topological Skew Braces]
{Solvability and Rigidity for Topological Skew Braces}

\author{Marco Damele}
\address{Dipartimento di Matematica \\
Universit\`a di Cagliari (Italy)}
\email{m.damele4@studenti.unica.it}

\author{Andrea Loi}
\address{Dipartimento di Matematica \\
Universit\`a di Cagliari (Italy)}
\email{loi@unica.it}

\thanks{
The authors are supported by INdAM and GNSAGA -- Gruppo Nazionale per le Strutture Algebriche, Geometriche e le loro Applicazioni.}



\begin{abstract}
We study compact and locally compact topological analogues of the
Byott--Vendramin solvability problem for finite skew braces: whether
solvability of the additive group forces solvability of the multiplicative
group. Our main result concerns the locally compact setting: if
$B=(B,\cdot,\circ)$ is a connected locally compact Hausdorff topological skew
brace and the additive group $(B,\cdot)$ is solvable, then the multiplicative
group $(B,\circ)$ is solvable. 
The proof combines a structural reduction to a solvable Lie quotient of the
additive group with an affine-action theorem showing that a connected Lie group
acting transitively and affinely on a connected solvable Lie group, with
solvable stabilizer identity component, is itself solvable.
We also show that the Hausdorff, local
compactness, and connectedness assumptions cannot be omitted in general.
Finally, in the compact connected Hausdorff case with abelian additive group,
we prove a stronger rigidity result: the two group laws coincide.
\end{abstract}

\maketitle

\tableofcontents
\section{Introduction}\label{sec:introduction}

Skew braces, introduced by Guarnieri and Vendramin \cite{Guarnieri2017},
provide a natural algebraic framework for the study of set-theoretic solutions
of the Yang--Baxter equation. Since their introduction, they have revealed deep
connections with group theory, Hopf--Galois structures, and various aspects of
solvability; see, for instance, \cite{Smoktunowicz2018,Nasybullov}.

A fundamental question in the theory concerns the interplay between the
algebraic properties of the two underlying groups of a skew brace. In
particular, Question~2.25 in \cite{Smoktunowicz2018}, originating from work of
Byott on Hopf--Galois structures, asks whether, for a finite skew brace,
solvability of the additive group forces solvability of the multiplicative
group. We refer to this as the Byott--Vendramin solvability problem. This
problem is open in general, although it has been verified in several special
cases; see, for instance, \cite{Nasybullov,GorshkovNasybullov}.

In this paper, we study topological analogues of the Byott--Vendramin
solvability problem in the setting of locally compact and compact skew braces.
More precisely, we consider skew braces whose underlying sets are endowed with
compatible topological group structures, and we analyze how topological
assumptions influence the interaction between the additive and multiplicative
groups.

The first main result of the paper is the following, proved as
Theorem~\ref{thm:lc-solvable}.

\begin{maintheorem}
Let $(B,\cdot,\circ)$ be a connected locally compact Hausdorff topological skew
brace. If the additive group $(B,\cdot)$ is solvable, then the multiplicative
group $(B,\circ)$ is solvable.
\end{maintheorem}

The proof relies on a reduction to Lie theory. Using the structure theory of
connected locally compact groups, we pass to a Lie quotient of the additive
group and interpret the skew brace structure in terms of affine actions. 
This allows us to apply the affine-action result proved in
Proposition~\ref{prop:affine-general}, which states that a connected Lie group
acting transitively and affinely on a connected solvable Lie group, with
solvable stabilizer identity component, is itself solvable.
The assumptions in Theorem~\ref{thm:lc-solvable} are essential. More precisely,
the Hausdorff assumption is addressed in Proposition~\ref{prop:indiscrete} and
Corollary~\ref{cor:nasybullov}, local compactness in
Proposition~\ref{prop:banach-counterexample}, and connectedness in
Theorem~\ref{thm:counterexample-compact}.

A second ingredient of the paper is a simple Pontryagin-duality rigidity
principle: a connected topological group cannot act non-trivially by continuous
automorphisms on a compact Hausdorff abelian group. This lemma is used both in
the proof of the locally compact solvability theorem and in the proof of the
following compact rigidity result, proved as
Theorem~\ref{thm:compact-rigidity}.

\begin{compactrigidity}
Let $(B,\cdot,\circ)$ be a compact connected Hausdorff topological skew brace.
If $(B,\cdot)$ is abelian, then
\[
a\circ b=a\cdot b
\qquad (a,b\in B).
\]
\end{compactrigidity}

Thus, in the compact connected Hausdorff case with abelian additive group, no
nontrivial topological skew brace structure exists.

The paper is organized as follows. In Section~\ref{sec:preliminaries} we recall
basic notions on topological skew braces and compact connected solvable groups.
In Section~\ref{sec:affine-actions} we prove Proposition~\ref{prop:affine-general}.
Section~\ref{sec:locally-compact} contains the Pontryagin-duality rigidity
lemma, the proof of the compact rigidity theorem, and the proof of
Theorem~\ref{thm:lc-solvable}. In Section~\ref{sec:necessity} we show that the
Hausdorff, local compactness, and connectedness assumptions cannot be omitted.

\section{Preliminaries on topological groups and skew braces}\label{sec:preliminaries}

We recall the basic definitions and several standard results that will be used throughout the paper. For the group-theoretic aspects of skew braces, we refer the reader to \cite{Guarnieri2017}; for the theory of Lie skew braces, see \cite{DameleLoi2026}.

\subsection{Topological skew braces}

A \emph{skew left brace} is a triple
\[
B=(B,\cdot,\circ),
\]
where $(B,\cdot)$ and $(B,\circ)$ are groups on the same underlying set, such
that
\[
a\circ (b\cdot c)
=
(a\circ b)\cdot a^{-1}\cdot (a\circ c)
\qquad (a,b,c\in B),
\]
where $a^{-1}$ denotes the inverse of $a$ in the group $(B,\cdot)$.
If $(B,\cdot)$ is abelian, one speaks simply of a \emph{left brace}.

A \emph{topological skew brace} is a skew brace
\[
B=(B,\cdot,\circ)
\]
such that both $(B,\cdot)$ and $(B,\circ)$ are
topological groups, with the same underlying topology.

The two group structures have the same identity element. We denote it by $e$.
Associated with every skew brace is the map
\[
\lambda\colon (B,\circ)\longrightarrow \Aut(B,\cdot),
\qquad
\lambda_a(b):=a^{-1}\cdot(a\circ b).
\]
The brace identity is equivalent to saying that each $\lambda_a$ is an
automorphism of $(B,\cdot)$ and that $\lambda$ is a group homomorphism. Moreover,
one has
\[
a\circ b=a\cdot \lambda_a(b)
\qquad (a,b\in B).
\]
If $B$ is a topological skew brace, then the action
\[
(B,\circ)\times (B,\cdot)\longrightarrow (B,\cdot),
\qquad
(a,b)\longmapsto \lambda_a(b),
\]
is continuous.

\subsection{Compact connected solvable groups}

\begin{lemma}\label{lem:compact-connected-solvable-abelian}
Let $K$ be a compact connected Hausdorff topological group. If $K$ is solvable,
then $K$ is abelian.
\end{lemma}

\begin{proof}
This is a standard theorem on compact connected groups; see
\cite[Theorem~9.33]{MorrisCompactGroups}.
\end{proof}

\begin{lemma}\label{lem:max-compact-connected-solvable}
Let $G$ be a connected locally compact Hausdorff solvable group, and let
\[
C:=\MaxK(G)
\]
be its largest compact normal subgroup. Then $C$ is characteristic and
connected. Consequently, $C$ is abelian and $G/C$ is a connected solvable Lie
group.
\end{lemma}

\begin{proof}
The existence of the largest compact normal subgroup \(C\), and the fact that
\(G/C\) is a connected Lie group, are standard in the structure theory of
connected locally compact groups; see
\cite[Theorem~8.36]{HofmannMorrisProLie}.

Since \(C\) is the largest compact normal subgroup of \(G\), every automorphism
of \(G\) preserves \(C\). Hence \(C\) is characteristic.
By the structure theory of connected locally compact solvable groups, the
maximal compact normal subgroup of \(G\) is connected; see
\cite[Theorem~8.37]{HofmannMorrisProLie}. Thus
$C=C^0$.
Since \(G\) is solvable, the subgroup \(C\) is solvable. Therefore, by
Lemma~\ref{lem:compact-connected-solvable-abelian}, the compact connected
solvable group \(C\) is abelian.
Finally, \(G/C\) is connected because it is the continuous image of the
connected group \(G\), and it is solvable because it is a quotient of the
solvable group \(G\). Hence \(G/C\) is a connected solvable Lie group.
\end{proof}

\section{Affine actions on solvable Lie groups}\label{sec:affine-actions}
The goal of this section is to prove Proposition~\ref{prop:affine-general},
which will be the Lie-theoretic input in the proof of
Theorem~\ref{thm:lc-solvable}.

In the sense of \cite{DameleLoi2026}, an affine action of a Lie group on
another Lie group is described by a homomorphism into the corresponding affine
group
\[
\Aff(L)=L\rtimes \Aut(L).
\]
We adopt this terminology throughout.

\begin{definition}\label{def:affine-action}
Let $H$ and $L$ be connected Lie groups. A \emph{smooth affine action} of $H$
on $L$ is a Lie group homomorphism
\[
\rho\colon H\longrightarrow \Aff(L)=L\rtimes \Aut(L).
\]
Writing
\[
\rho(h)=(t_h,\phi_h)\in L\rtimes \Aut(L),
\]
the induced action of $H$ on $L$ is given by
\[
h\cdot x:=t_h\cdot \phi_h(x)
\qquad (h\in H,\ x\in L).
\]
The action is called \emph{transitive} if for every $x,y\in L$ there exists
$h\in H$ such that $h\cdot x=y$, and \emph{simply transitive} if such an
element $h$ is unique.
\end{definition}

\begin{remark}\label{rem:damele-loi}
This agrees with the notion of affine action used by Damele--Loi \cite{DameleLoi2026}. In the
presence of a Lie skew brace, the associated affine action is simply transitive.
In the present section, however, we work in the more general setting of
transitive affine actions, allowing non-trivial stabilizers.
\end{remark}

\begin{remark}\label{rem:holonomy}
For an affine action
\[
\rho\colon H\to \Aff(L)=L\rtimes \Aut(L),
\]
we denote by
\[
\operatorname{hol}\colon \Aff(L)\longrightarrow \Aut(L),
\qquad
(t,\phi)\longmapsto \phi,
\]
the natural projection onto the automorphism factor. The subgroup
\[
\operatorname{hol}(\rho(H))=(\operatorname{hol}\circ \rho)(H)\subseteq \Aut(L)
\]
is called the \emph{holonomy group} of the action.
\end{remark}

\begin{definition}[Strong unipotent radical]
Let $H$ be a real linear algebraic group, and let $u(H)$ denote its unipotent
radical. We say that $H$ has a \emph{strong unipotent radical} if
\[
Z_H\bigl(u(H)\bigr)\subseteq u(H).
\]
\end{definition}

\begin{definition}[Real algebraic hull]\label{def:real-algebraic-hull}
Let $L$ be a connected simply connected solvable Lie group.
A \emph{real algebraic hull} of $L$ is a solvable-by-finite real linear
algebraic group $H_L$, i.e. a real linear algebraic group whose identity
component $H_L^\circ$ is solvable, with strong unipotent radical, together
with a continuous injective homomorphism
\[
L\hookrightarrow H_L
\]
whose image is Zariski-dense and such that
\[
\dim u(H_L)=\dim L.
\]\end{definition}

\begin{definition}[$d$-subgroups and algebraic automorphisms]
Let $H$ be a solvable-by-finite real linear algebraic group.
A \emph{$d$-subgroup} of $H$ is a Zariski-closed subgroup consisting of
semisimple elements. A \emph{maximal $d$-subgroup} is a $d$-subgroup maximal
with respect to inclusion. We write $\Aut_a(H)$ for the group of algebraic
automorphisms of $H$.
\end{definition}

\begin{lemma}[Baues' structural results]\label{lem:baues-input}
Let $L$ be a connected simply connected solvable Lie group, let
$L\hookrightarrow H_L$
be a fixed Zariski-dense inclusion into its real algebraic hull, and put
$U:=u(H_L)$.
Then the following hold.

\begin{enumerate}[label=\textup{(\alph*)}]
\item
If $T\leq H_L$ is a maximal $d$-subgroup, then $H_L=U\rtimes T$
and there is an algebraic projection
\[
\tau_T\colon H_L\longrightarrow U,
\qquad
\tau_T(ut)=u
\qquad (u\in U,\ t\in T).
\]

\item The subgroup $L$ is closed and normal in $H_L$, and for every maximal
$d$-subgroup $T\leq H_L$ one has $H_L=L\rtimes T$.
Moreover, the restriction of $\tau_T$ to $L$ is a diffeomorphism
\[
\tau:=\tau_T|_L\colon L\longrightarrow U.
\]

\item The natural extension homomorphism identifies $\Aut(L)$ with a
Zariski-closed subgroup of $\Aut_a(H_L)$.

\item
If $R\leq \Aut_a(H_L)$ is a reductive subgroup, then there exists a maximal
$d$-subgroup $T\leq H_L$ stabilized by $R$.

\item Let $\Delta\leq \Aff(L)$ be a subgroup. Assume that
$\operatorname{hol}(\Delta)$ stabilizes a maximal $d$-subgroup $T\leq H_L$.
Then the diffeomorphism
\[
\tau=\tau_T|_L\colon L\longrightarrow U
\]
conjugates the natural action of $\Delta$ on $L$ to an affine action on $U$.
\end{enumerate}

\begin{proof}
Part \textup{(a)} is \cite[Proposition~2.1]{Baues}.
Part \textup{(b)} is \cite[Proposition~2.3]{Baues}.
Part \textup{(c)} is \cite[Proposition~2.5]{Baues}.
Part \textup{(d)} is \cite[Lemma~3.7]{Baues}.
Part \textup{(e)} is \cite[Lemma~3.6]{Baues}.
\end{proof}
\end{lemma}

\begin{lemma}\label{lem:semisimple-zariski-reductive}
Let $S$ be a connected semisimple Lie group, let $G$ be a real linear
algebraic group, and let
\[
\psi\colon S\longrightarrow G
\]
be a continuous homomorphism. Then the Zariski closure
$\overline{\psi(S)}^{\,Z}$
is a connected reductive algebraic subgroup of $G$.
\end{lemma}

\begin{proof}
Choose a faithful algebraic representation
\[
\iota\colon G\hookrightarrow \GL(V),
\]
and set $\rho:=\iota\circ\psi$.
Let $H:=\overline{\rho(S)}^{\,Z}$.

Since $S$ is connected, $H$ is connected.
The representation $d\rho$ of $\Lie(S)$ is completely reducible, hence so is
$\rho$. It follows that $H$ acts semisimply on $V$, and therefore $H$ is
reductive. The claim follows.
\end{proof}

Auslander's theorem \cite{Auslander} shows that every Lie group admitting a
simply transitive affine action is solvable. The following proposition extends
this conclusion to transitive affine actions.

\begin{proposition}\label{prop:affine-general}
Let $L$ be a connected solvable Lie group, and let
\[
\rho\colon H\longrightarrow \Aff(L)
\]
be a smooth affine action of a connected Lie group $H$ on $L$.
Assume that the action is transitive and that the identity component of one
point stabilizer is solvable. Then $H$ is solvable.
\end{proposition}

\begin{proof}
Let
\[
p\colon \widetilde L\to L
\qquad\text{and}\qquad
\pi\colon \widetilde H\to H
\]
be the universal covering homomorphisms.

We first lift the affine action to the universal covers. Write
\[
\rho(h)=(t_h,\phi_h)\in L\rtimes \Aut(L).
\]
Every automorphism of \(L\) lifts uniquely to an automorphism of the simply
connected Lie group \(\widetilde L\) fixing the identity. Moreover, every left
translation of \(L\) lifts to a left translation of \(\widetilde L\). Since
\(\widetilde H\) is simply connected, the lifted affine transformations assemble
to a Lie group homomorphism
\[
\widetilde\rho\colon \widetilde H\longrightarrow \Aff(\widetilde L)
\]
such that
\[
p(\widetilde h\cdot \widetilde x)
=
\pi(\widetilde h)\cdot p(\widetilde x)
\qquad
(\widetilde h\in\widetilde H,\ \widetilde x\in\widetilde L).
\]
The lifted action is smooth, affine, and transitive.

Fix \(\widetilde x\in \widetilde L\), and put \(x=p(\widetilde x)\). Let
\(\widetilde H_{\widetilde x}\) and \(H_x\) denote the stabilizers of
\(\widetilde x\) and \(x\), respectively. The homomorphism \(\pi\) maps
\(\widetilde H_{\widetilde x}\) into \(H_x\), and hence induces
$\pi_{\widetilde x}\colon
\widetilde H_{\widetilde x}\longrightarrow H_x$.
Passing to identity components gives
\[
\pi_{\widetilde x}^0\colon
\widetilde H_{\widetilde x}^{\,0}
\longrightarrow
H_x^0.
\]
Since \(\ker\pi\) is discrete and central in \(\widetilde H\), the kernel of
\(\pi_{\widetilde x}^0\) is discrete and central in
\(\widetilde H_{\widetilde x}^{\,0}\). 
By assumption, $H_x^0$ is solvable. Since $\ker\pi$ is discrete, the induced
map on Lie algebras identifies
$\Lie(\widetilde H_{\widetilde x}^{\,0})$
with a Lie subalgebra of $\Lie(H_x^0)$. Hence
$\widetilde H_{\widetilde x}^{\,0}$ is solvable.

Let \(S\) be a Levi subgroup of \(\widetilde H\). We prove that \(S\) is
trivial. Since \(\widetilde L\) is connected, simply connected, and solvable,
it admits a real algebraic hull. Therefore we may apply
Lemma~\ref{lem:baues-input} to \(\widetilde L\).
Consider the restricted affine action
\[
\widetilde\rho|_S\colon S\longrightarrow \Aff(\widetilde L)
=
\widetilde L\rtimes \Aut(\widetilde L).
\]
Let
\[
K:=\operatorname{hol}(\widetilde\rho(S))
\subseteq \Aut(\widetilde L)
\]
be its holonomy group. By Lemma~\ref{lem:baues-input}\textup{(c)}, we view
\(K\) as a subgroup of \(\Aut_a(H_{\widetilde L})\), where
\(H_{\widetilde L}\) is the real algebraic hull of \(\widetilde L\).

By Lemma~\ref{lem:semisimple-zariski-reductive}, the Zariski closure
\[
R:=\overline{K}^{\,Z}
\subseteq \Aut_a(H_{\widetilde L})
\]
is a connected reductive algebraic subgroup. Hence, by
Lemma~\ref{lem:baues-input}\textup{(d)}, there exists a maximal \(d\)-subgroup
$T\leq H_{\widetilde L}$
which is stabilized by \(R\), and therefore by \(K\).

Let
\[
U:=u(H_{\widetilde L})
\]
be the unipotent radical of \(H_{\widetilde L}\). By
Lemma~\ref{lem:baues-input}\textup{(b)}, the projection associated with the
decomposition
\[
H_{\widetilde L}=U\rtimes T
\]
restricts to a diffeomorphism
\[
\tau\colon \widetilde L\longrightarrow U.
\]
By Lemma~\ref{lem:baues-input}\textup{(e)}, this diffeomorphism conjugates the
affine action of \(\widetilde\rho(S)\) on \(\widetilde L\) to an affine action
on \(U\). Thus we obtain a homomorphism
\[
\overline\rho\colon S\longrightarrow \Aff(U),
\qquad
\overline\rho(s)=\tau\circ\widetilde\rho(s)\circ\tau^{-1}.
\]

Let
\[
R_0:=\overline{\overline\rho(S)}^{\,Z}
\leq \Aff(U)
\]
be the Zariski closure of \(\overline\rho(S)\). Again by
Lemma~\ref{lem:semisimple-zariski-reductive}, the group \(R_0\) is connected
and reductive. By \cite[Lemma~5.3]{Dere}, every reductive subgroup of
\(\Aff(U)\) fixes a point of \(U\). Therefore there exists \(u_0\in U\) such
that
\[
\overline\rho(s)(u_0)=u_0
\qquad
(s\in S).
\]

Set
\[
\widetilde x_0:=\tau^{-1}(u_0)\in \widetilde L.
\]
Since \(\tau\) conjugates the two actions, \(S\) fixes \(\widetilde x_0\).
In other words,
\[
s\cdot \widetilde x_0=\widetilde x_0
\qquad (s\in S).
\]
Hence
$S\leq \widetilde H_{\widetilde x_0}^{\,0}$.
But the identity component of every stabilizer for the lifted action is
solvable. Therefore \(S\) is solvable. Since \(S\) is also semisimple, it must
be trivial.
Hence
\(\widetilde H\) has trivial Levi subgroup, and therefore \(\widetilde H\) is
solvable. Finally, \(H\) is a quotient of \(\widetilde H\), so \(H\) is
solvable.
\end{proof}

\begin{remark}
In the application to Theorem~\ref{thm:lc-solvable}, the stabilizer is in fact
abelian. The more general solvable-stabilizer formulation of
Proposition~\ref{prop:affine-general} is nevertheless natural, since the proof
only requires that any Levi subgroup contained in a stabilizer be trivial.
\end{remark}

\section{The locally compact case and compact rigidity}\label{sec:locally-compact}

We now apply the affine-action result of the previous section to a suitable
Lie quotient of the additive group. We first clarify why such a quotient is
needed.

\begin{remark}
A connected locally compact Hausdorff topological group need not be a Lie
group; for instance,
\[
\prod_{n\ge 1}\mathbb T
\]
is compact, connected and Hausdorff, but not locally Euclidean.

By contrast, additional hypotheses may force the Lie property. For example,
a connected locally compact Hausdorff group admitting a faithful continuous
finite-dimensional linear representation is Lie. Similarly, a non-trivial
topologically simple connected locally compact Hausdorff group is Lie, by the
Gleason--Yamabe theorem.
\end{remark}

In the proof of our main result (Theorem  \ref{thm:lc-solvable})
we need the following lemma.

\begin{lemma}\label{lem:pontryagin-rigidity}
Let $G$ be a connected topological group, let $A$ be a compact Hausdorff
abelian topological group, and let
\[
\alpha\colon G\times A\longrightarrow A
\]
be a continuous action of $G$ on $A$ by topological automorphisms. Then the
action is trivial.
\end{lemma}

\begin{proof}
Let $\widehat A$ be the Pontryagin dual of $A$. Since $A$ is compact Hausdorff
abelian, $\widehat A$ is discrete, and continuous characters separate points of
$A$ \cite[Chapter~I]{MorrisCompactGroups}.

Fix $\chi\in\widehat A$. The map
\[
G\longrightarrow \widehat A,
\qquad
g\longmapsto \chi\circ \alpha_g,
\]
is continuous, where $\widehat A$ is viewed as a discrete subgroup of
$C(A,\T)$ with the compact-open topology. Since $G$ is connected, this map is
constant. Evaluating at the identity gives
\[
\chi\circ \alpha_g=\chi
\qquad
(g\in G).
\]
Since characters separate points, $\alpha_g=\id_A$ for every $g\in G$.
\end{proof}

As an immediate consequence of this rigidity lemma, we obtain the following
compact connected Hausdorff rigidity result.

\begin{theorem}\label{thm:compact-rigidity}
Let $B=(B,\cdot,\circ)$ be a compact connected Hausdorff topological skew brace.
Assume that the additive group $(B,\cdot)$ is abelian. Then
\[
a\circ b=a\cdot b
\qquad
(a,b\in B).
\]
In particular, $(B,\circ)$ is abelian, and the two topological group structures
coincide.
\end{theorem}

\begin{proof}
Since $(B,\cdot)$ is abelian, the lambda-map
\[
\lambda\colon (B,\circ)\longrightarrow \Aut(B,\cdot),
\qquad
\lambda_a(b)=a^{-1}\cdot(a\circ b),
\]
is a continuous action of $(B,\circ)$ on the compact Hausdorff abelian group
$(B,\cdot)$ by topological automorphisms. Moreover,
\[
a\circ b=a\cdot\lambda_a(b)
\qquad
(a,b\in B).
\]

Since $(B,\circ)$ is connected, Lemma~\ref{lem:pontryagin-rigidity} implies
that
\[
\lambda_a=\id_{(B,\cdot)}
\qquad
(a\in B).
\]
Therefore
\[
a\circ b=a\cdot\lambda_a(b)=a\cdot b
\qquad
(a,b\in B).
\]
\end{proof}

\begin{remark}
The compactness assumption in Theorem~\ref{thm:compact-rigidity} is essential:
the example constructed in Proposition~\ref{prop:banach-counterexample} is
connected and Hausdorff, has abelian additive group, but has non-solvable
multiplicative group. 
\end{remark}

We are now ready to prove the main result of the paper.

\begin{theorem}\label{thm:lc-solvable}
Let $B=(B,\cdot,\circ)$ be a connected locally compact Hausdorff topological
skew brace. If the additive group $(B,\cdot)$ is solvable, then the
multiplicative group $(B,\circ)$ is solvable.
\end{theorem}

\begin{proof}
Let \(C\) be the largest compact normal subgroup of \((B,\cdot)\).
By Lemma~\ref{lem:max-compact-connected-solvable}, \(C\) is compact,
connected, characteristic in \((B,\cdot)\), abelian, and
\[
L:=(B,\cdot)/C
\]
is a connected solvable Lie group.

Since \(C\) is characteristic in \((B,\cdot)\), every automorphism
\(\lambda_a\) preserves \(C\). Hence each \(\lambda_a\) induces an automorphism
\[
\overline{\lambda}_a\in \Aut(L),
\qquad
\overline{\lambda}_a(xC)=\lambda_a(x)C.
\]
Define
\[
j\colon (B,\circ)\longrightarrow \Aff(L),
\qquad
j(a)=\bigl(aC,\overline{\lambda}_a\bigr).
\]
Equivalently, \(j(a)\) acts on \(L\) by
\[
j(a)(xC)=aC\cdot \overline{\lambda}_a(xC)=(a\circ x)C.
\]
This is well defined because \(C\) is \(\lambda\)-invariant. Moreover, the map
\(j\) is a continuous homomorphism, since the translation part
\(a\mapsto aC\) and the automorphism part \(a\mapsto\overline{\lambda}_a\) are
continuous.

Let
\[
N:=\ker j.
\]
Since \(j\) is continuous and \(\Aff(L)\) is Hausdorff, \(N\) is closed in
\((B,\circ)\). Moreover, if \(a\in N\), then
\[
eC=j(a)(eC)=aC,
\]
so \(a\in C\). Hence \(N\subseteq C\).

We now show that \(N\) is abelian. Since \(C\) is characteristic in
\((B,\cdot)\), the restriction of the lambda map gives a continuous action
\[
(B,\circ)\times C\longrightarrow C,
\qquad
(a,c)\longmapsto \lambda_a(c),
\]
of the connected group \((B,\circ)\) on the compact Hausdorff abelian group
\(C\) by topological automorphisms. By Lemma~\ref{lem:pontryagin-rigidity},
this action is trivial. Thus
\[
\lambda_a|_C=\id_C
\qquad (a\in B).
\]
In particular, for \(c,d\in C\),
\[
c\circ d=c\cdot \lambda_c(d)=c\cdot d.
\]
Therefore \(C\) is a subgroup of \((B,\circ)\), the two group laws coincide on
\(C\), and \(C\) is abelian as a subgroup of \((B,\circ)\). Since
\(N\subseteq C\), it follows that \(N\) is abelian.

Put
\[
Q:=(B,\circ)/N.
\]
Then \(Q\) is a connected locally compact Hausdorff group, and \(j\) induces a
continuous injective homomorphism
\[
\overline{\rho}\colon Q\longrightarrow \Aff(L).
\]
The induced action of \(Q\) on \(L\) is given by
\[
aN\ast xC=(a\circ x)C.
\]
It is faithful and continuous. It is also transitive, because
\[
aN\ast eC=(a\circ e)C=aC,
\]
and every element of \(L\) is of the form \(aC\).

Since \(L\) is a connected Lie group, it is locally contractible. Therefore, by
Theorem~B of Hofmann--Kramer \cite{HofmannKramer}, the group \(Q\) is a Lie
group. Since both \(Q\) and \(\Aff(L)\) are Lie groups, the continuous
homomorphism
\[
\overline{\rho}\colon Q\longrightarrow \Aff(L)
\]
is automatically smooth. Hence \(\overline{\rho}\) is a smooth transitive
affine action of the connected Lie group \(Q\) on the connected solvable Lie
group \(L\).

We compute the stabilizer of
\[
o:=eC\in L.
\]
If \(aN\in Q\) fixes \(o\), then
\[
o=aN\ast o=(a\circ e)C=aC.
\]
Thus \(aC=eC\), so \(a\in C\). Conversely, if \(a\in C\), then
\[
aN\ast o=(a\circ e)C=aC=eC=o.
\]
Therefore the stabilizer of \(o\) in \(Q\) is exactly \(C/N\). Since \(C\) is
abelian as a subgroup of \((B,\circ)\), the group \(C/N\) is abelian, hence
solvable. In particular, the identity component of one point stabilizer is
solvable.

By Proposition~\ref{prop:affine-general}, \(Q\) is solvable. Finally, from the
exact sequence
\[
1\longrightarrow N
\longrightarrow (B,\circ)
\longrightarrow Q
\longrightarrow 1
\]
with abelian kernel and solvable quotient, we conclude that \((B,\circ)\) is
solvable.
\end{proof}

\section{Necessity of the assumptions}\label{sec:necessity}

We now show that the hypotheses in Theorem~\ref{thm:lc-solvable} are essential.
More precisely, the conclusion may fail if one drops Hausdorffness, local
compactness, or connectedness.

\subsection{Failure without the Hausdorff assumption}\label{subsec:non-hausdorff}

Without the Hausdorff assumption, even the compact connected topological
analogue is false.

\begin{proposition}\label{prop:indiscrete}
Let $B=(B,\cdot,\circ)$ be an abstract skew brace such that $(B,\cdot)$ is
solvable and $(B,\circ)$ is not solvable. Endow $B$ with the indiscrete
topology
\[
\tau=\{\varnothing,B\}.
\]
Then $(B,\tau,\cdot,\circ)$ is a compact connected non-Hausdorff topological
skew brace.
\end{proposition}

\begin{proof}
With the indiscrete topology every map from $B$ to $B$ is continuous. Hence both
group operations and both inversion maps are continuous, so
$(B,\tau,\cdot,\circ)$ is a topological skew brace.

The space $(B,\tau)$ is compact, because every open cover contains $B$ itself,
and it is connected, because it has no non-trivial clopen subsets. If
$|B|>1$, it is not Hausdorff. The underlying algebraic groups are unchanged, so
$(B,\cdot)$ remains solvable while $(B,\circ)$ remains non-solvable.
\end{proof}

\begin{corollary}\label{cor:nasybullov}
The non-Hausdorff compact connected topological Byott--Vendramin solvability
problem has a concrete counterexample obtained from Nasybullov's brace.
\end{corollary}

\begin{proof}
Nasybullov constructs a brace
\[
A=\bigoplus_{n\ge 2} A_n
\]
with additive group
\[
(A,\cdot)\cong \bigoplus_{\mathbb N}\Z,
\]
which is abelian, and multiplicative group
\[
(A,\circ)\cong \bigoplus_{n\ge 2} UT_n(\Z).
\]
The latter is not solvable, because the derived lengths of the unitriangular
groups $UT_n(\Z)$ are unbounded \cite[Example~3.2]{Nasybullov}. Applying
Proposition~\ref{prop:indiscrete} gives the desired compact connected
non-Hausdorff counterexample.
\end{proof}

\begin{remark}
This construction does not produce a Lie skew brace in the sense of
Damele--Loi \cite{DameleLoi2026}. A Lie skew brace is defined on a real smooth manifold carrying two
compatible Lie group structures \cite{DameleLoi2026}; in particular, its underlying
space is Hausdorff and locally Euclidean.
\end{remark}

\subsection{Failure without local compactness}\label{subsec:non-locally-compact}

We next show that local compactness cannot be omitted, even in the Hausdorff
connected case and even when the additive group is abelian.

\begin{proposition}\label{prop:banach-counterexample}
There exists a connected Hausdorff topological left brace
\[
B=(B,+,\circ)
\]
such that $(B,+)$ is abelian, the underlying space $B$ is not locally compact,
and $(B,\circ)$ is not solvable.
\end{proposition}

\begin{proof}
For each $n\ge 2$, let
\[
J_n:=\{A\in M_n(\R): A \text{ is strictly upper triangular}\},
\]
endowed with the operator norm. Consider the Banach algebra
\[
R:=c_0(J_n)
=
\Bigl\{x=(x_n)_{n\ge 2}: x_n\in J_n,\ \|x_n\|\to 0\Bigr\}
\]
with coordinatewise operations and supremum norm. Define
\[
x\circ y:=x+y+xy.
\]
Then $(R,+)$ is abelian.
Let $R^\sharp=\R 1\oplus R$ be the unitization. The map $x\mapsto 1+x$
identifies $(R,\circ)$ with $1+R\subseteq R^\sharp$. We show that every element
of $1+R$ is invertible. Given $x=(x_n)\in R$, write $x=y+z$, where $y$ has
finite support and $z$ is a sufficiently small tail. Then $1+y$ is invertible,
because $y$ is nilpotent, and choosing the tail so that
\[
\|(1+y)^{-1}z\|<1
\]
shows by the Neumann series that
\[
1+x=(1+y)\bigl(1+(1+y)^{-1}z\bigr)
\]
is invertible. Hence $(R,\circ)$ is a group.

The operation $\circ$ is continuous, and inversion is continuous because
$1+R$ is contained in the group of invertible elements of the unital Banach
algebra $R^\sharp$ \cite[Chapter~2]{BonsallDuncan}. The brace identity follows
from distributivity:
\[
x\circ (y+z)=(x\circ y)-x+(x\circ z).
\]
Thus $(R,+,\circ)$ is a topological left brace.
The Banach space $R$ is Hausdorff and path-connected, hence connected. It is
infinite-dimensional, and therefore not locally compact.
Finally, for each $n\ge 2$, the subgroup of elements supported only in the
$n$-th coordinate is isomorphic, under $x\mapsto I_n+x_n$, to $UT_n(\mathbb R)$.
It is well known, and follows for instance from the description of the derived
series of unitriangular groups, that
\[
\operatorname{dl}(UT_n(\mathbb R))=\lceil \log_2 n\rceil .
\]
Hence these subgroups have unbounded derived length. If $(R,\circ)$ were
solvable, then all its subgroups would have derived length bounded by
$\operatorname{dl}(R,\circ)$, a contradiction. Thus $(R,\circ)$ is not solvable.

\end{proof}

\begin{remark}
The brace constructed in Proposition~\ref{prop:banach-counterexample} is
naturally an infinite-dimensional Banach--Lie left brace. However, it is not a
finite-dimensional Lie skew brace, and in particular it is not locally compact.
\end{remark}

\subsection{Failure without connectedness}\label{subsec:non-connected}

Finally, we show that connectedness cannot be omitted, even in the compact
Hausdorff case.

\begin{theorem}\label{thm:counterexample-compact}
There exists a compact Hausdorff totally disconnected topological two-sided
brace
\[
B=(B,+,\circ)
\]
such that $(B,+)$ is abelian, while $(B,\circ)$ is prosolvable but not
solvable.
\end{theorem}

\begin{proof}
Fix a prime number $p$. For every $n\ge 2$, let $J_n(\mathbb F_p)$ be the ring
of strictly upper triangular $n\times n$ matrices over $\mathbb F_p$. Since
$J_n(\mathbb F_p)$ is nilpotent, it is Jacobson radical, and therefore its
additive group together with the adjoint operation
\[
x\circ y:=x+y+xy
\]
forms a finite two-sided brace. Denote this brace by
\[
A_n=(J_n(\mathbb F_p),+,\circ).
\]
The map
\[
A_n\longrightarrow UT_n(\mathbb F_p),
\qquad
x\longmapsto I_n+x,
\]
is an isomorphism from $(A_n,\circ)$ onto $UT_n(\mathbb F_p)$.

Set
\[
B:=\prod_{n\ge 2} A_n
\]
with the product topology and coordinatewise operations. Since each $A_n$ is
finite and discrete, $B$ is compact, Hausdorff, and totally disconnected. Its
additive group is abelian.

Moreover,
\[
(B,\circ)\cong \prod_{n\ge 2}UT_n(\mathbb F_p).
\]
Each factor is finite and solvable, so $(B,\circ)$ is prosolvable. However, the
derived length of the unitriangular group satisfies
\[
\operatorname{dl}(UT_n(\mathbb F_p))=\lceil\log_2 n\rceil .
\]
Thus the derived lengths of the finite factors are unbounded. If the direct
product were solvable of derived length $d$, then every factor would have
derived length at most $d$, a contradiction. Hence $(B,\circ)$ is not solvable.
\end{proof}

\begin{remark}
Since $B$ is totally disconnected, its identity component is trivial. Thus the
connectedness assumption in Theorem~\ref{thm:lc-solvable} is essential.
\end{remark}


\begin{thebibliography}{99}

\bibitem{Auslander}
L.~Auslander,
\emph{Simply transitive groups of affine motions},
Amer. J. Math. \textbf{99} (1977), no.~4, 809--826.

\bibitem{Baues}
O.~Baues,
\emph{Infra-solvmanifolds and rigidity of subgroups in solvable linear algebraic groups},
Topology \textbf{43} (2004), no.~4, 903--924.

\bibitem{BonsallDuncan}
F.~F.~Bonsall and J.~Duncan,
\emph{Complete Normed Algebras},
Ergebnisse der Mathematik und ihrer Grenzgebiete, vol.~80,
Springer-Verlag, Berlin--Heidelberg--New York, 1973.



\bibitem{Dere}
J.~Der\'e,
\emph{NIL-affine crystallographic actions of virtually polycyclic groups},
Transform. Groups \textbf{26} (2021), no.~4, 1217--1240.

\bibitem{DameleLoi2026}
M.~Damele and A.~Loi,
\emph{Structural and rigidity properties of Lie skew braces},
J. Algebra \textbf{695} (2026), 356--383.
\href{https://doi.org/10.1016/j.jalgebra.2026.01.048}{doi:10.1016/j.jalgebra.2026.01.048}.


\bibitem{GorshkovNasybullov}
I.~Gorshkov and T.~Nasybullov,
\emph{Finite skew braces with solvable additive group},
J. Algebra \textbf{574} (2021), 172--183.

\bibitem{Guarnieri2017}
L.~Guarnieri and L.~Vendramin,
\emph{Skew braces and the Yang--Baxter equation},
Math. Comp. \textbf{86} (2017), no.~307, 2519--2534.
\href{https://doi.org/10.1090/mcom/3161}{doi:10.1090/mcom/3161}.

\bibitem{HofmannKramer}
K.~H.~Hofmann and L.~Kramer,
\emph{Transitive actions of locally compact groups on locally contractible spaces},
J. Reine Angew. Math. \textbf{702} (2015), 227--243.

\bibitem{HofmannMorrisProLie}
K.~H.~Hofmann and S.~A.~Morris,
\emph{The Lie Theory of Connected Pro-Lie Groups. A Structure Theory for Pro-Lie Algebras, Pro-Lie Groups, and Connected Locally Compact Groups},
EMS Tracts in Mathematics, vol.~2, European Mathematical Society, Z\"urich, 2007.



\bibitem{MorrisCompactGroups}
K.~H.~Hofmann and S.~A.~Morris,
\emph{The Structure of Compact Groups}, 3rd ed.,
De Gruyter Studies in Mathematics, vol.~25, Walter de Gruyter, Berlin, 2013.


\bibitem{Nasybullov}
T.~Nasybullov,
\emph{Connections between properties of the additive and the multiplicative groups of a two-sided skew brace},
J. Algebra \textbf{540} (2019), 156--167.

\bibitem{Smoktunowicz2018}
A.~Smoktunowicz and L.~Vendramin,
\emph{On skew braces},
J. Comb. Algebra \textbf{2} (2018), no.~1, 47--86.
\href{https://doi.org/10.4171/JCA/2-1-3}{doi:10.4171/JCA/2-1-3}



\end{thebibliography}
\end{document}